\documentclass[12pt,a4paper,oneside,reqno]{article}

\usepackage[utf8]{inputenc}
\usepackage[english]{babel}
\usepackage[left=1in,right=1in,top=1in,bottom=1in]{geometry}
\usepackage{amsfonts,amssymb,amsmath,amsthm}
\usepackage[symbol]{footmisc}
\usepackage{longtable,array}
\newcolumntype{L}[1]{>{\raggedright\arraybackslash$}p{#1}<{$}}
\usepackage{enumerate}
\usepackage{cite}
\usepackage{xcolor}
\usepackage[normalem]{ulem}
\usepackage{tikz}
\usetikzlibrary{arrows.meta}
\newcommand{\add}[1]{#1}
\newcommand{\mdel}[1]{#1}
\newcommand{\C}{\mathbb{C}}
\usepackage{cancel}
\usepackage{tikz}
\usepackage{adjustbox}
\usetikzlibrary{arrows}
\usepackage{hyperref}
\hypersetup{hidelinks}
\usepackage{fancyhdr}
\theoremstyle{plain}
\newtheorem{theorem}{Theorem}
\newtheorem{lemma}[theorem]{Lemma}
\newtheorem{proposition}[theorem]{Proposition}
\newtheorem{remark}[theorem]{Remark}

\newtheorem{corollary}[theorem]{Corollary}

\numberwithin{theorem}{section}
\numberwithin{equation}{section}

\DeclareMathOperator{\tr}{tr}
\DeclareMathOperator{\Der}{Der}
\DeclareMathOperator{\Ann}{Ann}

\IfFileExists{fouriernc.sty}{\usepackage{fouriernc}}{}

\newcommand{\doi}[1]{\href{https://doi.org/#1}{doi:#1}}
\newcommand{\arxiv}[1]{\href{https://arxiv.org/abs/#1}{arXiv:#1}}
\newcommand{\arxivsearch}[1]{\href{https://arxiv.org/search/?query=#1}{arXiv}}
\tikzset{newedge/.style={}}

\begin{document}

\noindent\fbox{\parbox{\dimexpr\textwidth-2\fboxsep-2\fboxrule\relax}{\small
\textbf{Corrected version.} This is a corrected version of the paper I.~Kaygorodov, O.~Shashkov,
\emph{Degenerations of nilalgebras}, 
Communications in Mathematical Research, 41 (2025), no.~1, 9--24,
\doi{10.4208/cmr.2024-0033}.}}

\bigskip
  
\begin{center}
{\Large Degenerations of nilalgebras}

\bigskip

{\bf Ivan Kaygorodov and Oleg Shashkov}
\end{center}

\medskip

\noindent{\small\textbf{Abstract.} All complex 3-dimensional nilalgebras were described. As a corollary, all degenerations in the variety of complex 3-dimensional nilalgebras were obtained.}

\bigskip

\noindent {\bf Keywords:}
{\it Nilalgebra, algebraic classification, geometric classification, degeneration.}

\bigskip

\noindent {\bf MSC2020:} 17A30 (primary); 14L30 (secondary).

\bigskip

\tableofcontents

\section*{Introduction}
\addcontentsline{toc}{section}{Introduction}

An element $x$ is nil, if there exists a number $n$, such that for each $k\geq n$ we have $x^k=0$\footnote{By $x^k$ we mean all possible arrangements of non-associative products.}. An algebra is called a nilalgebra if each element is nil. The class of nilalgebras plays an important role in the ring theory. So, K\"othe's problem is one of the old problems in ring and module theory that has not yet been solved. A problem of the existence of simple associative nil rings was actualized by Kaplansky and successfully solved by Smoktunowicz \cite{b15}. Another famous problem was posted by Albert: Is every finite-dimensional (commutative) power associative nilalgebra solvable? It is still open, but it was solved in some particular cases \cite{b14}. In the present note, we give a positive answer to the problem of Albert for non-anticommutative $3$-dimensional algebras. Let us note, that in the anticommutative case, the problem of Albert does not make sense: each anticommutative algebra is nilalgebra with nilindex $2$ and in almost all dimensions there are simple Lie algebras, that are not solvable. On the other hand, for each $n>3$, Correa and Hentzel constructed a non-(anti)commutative $n$-dimensional non-solvable nilalgebra \cite{b3}. To do it, we obtain the full classification of complex $3$-dimensional nilalgebras and as a result, we have a geometric classification and the description of all degenerations in the variety of complex $3$-dimensional nilalgebras. In particular, we proved that this variety of algebras has two rigid algebras. Let us note, that the classification of $4$-dimensional commutative nilalgebras is given in \cite{b4}. It is known that each right Leibniz algebra (i.e., an algebra satisfying the identity $x(yz)=(xy)z+y(xz)$; about Leibniz algebras see, for example, \cite{b11} and references therein) satisfies the identity $x^2x=0$. Hence, the variety of symmetric Leibniz algebras (i.e., left and right Leibniz algebras) gives a subvariety in the variety of nilalgebras. Thanks to \cite{b1}, the intersection of right mono Leibniz (i.e. algebras where each one-generated subalgebra is a right Leibniz algebra) and left mono Leibniz algebras gives the variety of nilalgebras with nilindex $3$. As one more corollary from our result, we have the algebraic and geometric classification of $3$-dimensional symmetric mono Leibniz algebras.

\setcounter{section}{1}\section{The algebraic classification of 3-dimensional nilalgebras}

\subsection{Nilalgebras with nilindex 3}
By identity $x^3=0$ we mean the system of two identities $x^2x=0$ and $xx^2=0$. Linearizing them, we obtain a pair of useful identities
\begin{equation}\label{lin}
(xy+yx)x=-x^2y \quad\text{and}\quad x(xy+yx)=-yx^2 .
\end{equation}
The full linearization gives the following two identities:
\[
\sum_{\sigma\in\mathbb S_3}(x_{\sigma(1)}x_{\sigma(2)})x_{\sigma(3)}=0
\quad\text{and}\quad
\sum_{\sigma\in\mathbb S_3}x_{\sigma(1)}(x_{\sigma(2)}x_{\sigma(3)})=0 .
\]
We aim to classify all complex $3$-dimensional algebras $A$ satisfying $x^3=0$. Obviously, each anticommutative algebra has this property. These algebras were classified in \cite{b7}. Hence, we will consider only non-anticommutative cases, i.e. algebras where there is an element $a$, such that $a^2\neq 0$. It is easy to see, that $a$ and $a^2$ are linearly independent and (\ref{lin}) gives $a^2a^2=0$. Consider a basis $\{a,a^2,b\}$ of the algebra $A$, then we have the following statement.

\begin{lemma}\label{l1}
For the algebra $A$ with a basis $\{a,a^2,b\}$ there are the following three cases: $(T_0)$: $ba=0$, $(T_a)$: $ba=a$, $(T_b)$: $ba=b$.
\end{lemma}
\begin{proof}
Let $ba=\lambda a+\mu a^2+\nu b$. Suppose first that $\nu\neq0$. By using the substitution $a/\nu\mapsto a$, we can suppose that $\nu=1$. Let $b'=b+\lambda a+(\lambda+\mu)a^2$, then
\[ b'a=\lambda a+\mu a^2+b+\lambda a^2=b'. \]
Thus, for $\nu\neq0$ the algebra $A$ by substitution $b'\mapsto b$ is of type $(T_b)$. Let now $\nu=0$. Denote $b'=-\mu a+b$, then $b'a=-\mu a^2+\lambda a+\mu a^2=\lambda a$. Thus, for $\lambda=0$ the algebra $A$ by the substitution $b'\mapsto b$ is of type $(T_0)$, and for $\lambda\neq0$ by the substitution $b'/\lambda\mapsto b$ of type $(T_a)$.
\end{proof}

\subsubsection{Commutative nilalgebras with nilindex 3}
\begin{lemma}\label{l2}
If the algebra $A$ is commutative, then in some basis $\{a,a^2,b\}$ it is determined by one of the multiplication tables
\[
C_{0,0}:\ ab=a^2b=0,\ b^2=0, \qquad C_{0,a^2}:\ ab=a^2b=0,\ b^2=a^2 .
\]
\end{lemma}
\begin{proof}
Using the identity (\ref{lin}) and commutativity, we have $a^2b=-2(ba)a$, $b^2a=-2(ab)b$. Hence, we have the following relations in the three cases considered:
\begin{gather*}
(T_0):\ a^2b=0,\ b^2a=0;\qquad (T_a):\ a^2b=-2a^2,\ b^2a=-2a;\\
(T_b):\ a^2b=-2b,\ b^2a=-2b^2 .
\end{gather*}
It remains to determine $b^2$ in each of the three cases. Let us suppose that $b^2=\lambda a+\mu a^2+\nu b$.

\textbf{Case $(T_0)$.} First, $b^2a=0$. On the other hand, $b^2a=\lambda a^2$. Therefore, $\lambda=0$ and $b^2=\mu a^2+\nu b$. Further, $b^2b=0$ and $b^2b=\nu b^2=\nu\mu a^2+\nu^2b$. So $\nu=0$ and $b^2=\mu a^2$. Thus, we get two cases: either $b^2=0$, or (after substituting $b/\sqrt\mu\mapsto b$) $b^2=a^2$. So, in this case we have two algebras: $C_{0,0}$ and $C_{0,a^2}$.

\textbf{Case $(T_a)$.} First note that $b^2a=-2a$. On the other hand, $b^2a=\lambda a^2+\nu a$. So $b^2=\mu a^2-2b$. Further, $0=b^2b=\mu a^2b-2b^2=-2\mu a^2-2\mu a^2+4b$. Thus, the case $(T_a)$ is not realized in commutative algebras.

\textbf{Case $(T_b)$.} We have $b^2a=\lambda a^2+\nu b$ and $b^2a=-2(\lambda a+\mu a^2+\nu b)$. Then $\lambda=2\mu=3\nu=0$, i.e. $b^2=0$. Further, if $x=a^2+b$, then $x^2=2a^2b=-4b$, i.e. $0=x^2x=8b$. Thus, the case $(T_b)$ is neither realized in commutative algebras.
\end{proof}
It is easy to check that $\dim\Ann C_{0,0}=2$ and $\dim\Ann C_{0,a^2}=1$. Hence, they are non-isomorphic.

\subsubsection{Non-commutative nilalgebras with nilindex 3}
Let us consider the algebra $A^+$ with the multiplication $x\cdot y=\frac12(xy+yx)$. It is easy to see, that $A^+$ is commutative and satisfies the identity $x^3=0$.

\begin{lemma}\label{l3}
If $A$ is a non-commutative nilalgebra with nilindex $3$ and $A^+$ has type $C_{0,0}$ or $C_{0,a^2}$ regarding a basis $\{a,a^2,b\}$, then it satisfies: (1) $b^2=0$ or $b^2=a^2$; (2) $ba=-ab$; (3) $a^2a^2=a^2b=ba^2=0$.
\end{lemma}
\begin{proof}
It is easy to see, that $a\cdot a=a^2$ and $b^2=b\cdot b=c$, where $c=0$ or $c=a^2$. From $a\cdot b=0$ follows $ab=-ba$. The relation (\ref{lin}) gives $a^2b=-2(a\cdot b)a=0$ and on the other hand, we have $a^2b+ba^2=2(a\cdot a)\cdot b=0$, then $ba^2=-a^2b=0$. In the end, we obtain $a^2a^2=(a\cdot a)\cdot(a\cdot a)=0$.
\end{proof}
To summarize, for a description of the multiplication table of $A$ in the basis $\{a,a^2,b\}$, we have to determine the value of $ab$. Let $ab=\lambda a+\mu a^2+\nu b$. Let us define by $N_{c,d}$ a $3$-dimensional algebra with a basis $\{a,a^2,b\}$ and the multiplication table given below: $aa=a^2$, $ab=-ba=c$, $b^2=d$ (all other products of basis elements are zero).

\begin{proposition}\label{p1}
If $A^+\cong C_{0,0}$, then $A$ is isomorphic to $C_{0,0}$, $N_{a,0}$, $N_{b,0}$, or $N_{a^2,0}$.
\end{proposition}
\begin{proof}
Let $A$ be an algebra with a basis $\{a,a^2,b\}$ satisfying $ab=-ba=\lambda a+\mu a^2+\nu b$, $b^2=0$, and let $A'$ be an algebra with a basis $\{a',a'^2,b'\}$ and $a'b'=-b'a'=\lambda'a'+\mu'a'^2+\nu'b'$, $b'^2=0$. We suppose that $A'\cong A$ and let $\xi(a')=\alpha_1a+\beta_1a^2+\gamma_1b$, $\xi(b')=\alpha_2a+\beta_2a^2+\gamma_2b$ be an isomorphism. Then $0=\xi(b'^2)=\xi(b')^2=\alpha_2^2a^2$ and $\alpha_2=0$. Next, $\xi(a'^2)=\alpha_1^2a^2$ and
\[
\xi(a')\xi(b')=\alpha_1\gamma_2\,ab=\alpha_1\gamma_2(\lambda a+\mu a^2+\nu b),\quad
\xi(a'b')=\lambda'\xi(a')+\mu'\alpha_1^2a^2+\nu'\xi(b').
\]
Since $\dim\operatorname{Im}\xi=3$, we have $\alpha_1\gamma_2\neq0$, and comparing coefficients we obtain
\[
\lambda'=\gamma_2\lambda,\qquad \nu'=\alpha_1\nu-\gamma_1\lambda,\qquad
\mu'\alpha_1^2=\gamma_2(\alpha_1\mu-\beta_1\lambda)-\beta_2(\alpha_1\nu-\gamma_1\lambda).
\]
(1) If $\lambda\neq0$, choose $\gamma_2=1/\lambda$, $\gamma_1=\alpha_1\nu/\lambda$, $\beta_1=\alpha_1\mu/\lambda$, $\beta_2=0$; then $\lambda'=1$, $\nu'=\mu'=0$, and we obtain $N_{a,0}$.

(2) If $\lambda=0$, then $\lambda'=0$, $\nu'=\alpha_1\nu$, $\mu'\alpha_1=\gamma_2\mu-\beta_2\nu$. (a) If $\nu\neq0$, choose $\alpha_1=1/\nu$, $\beta_2=\gamma_2\mu/\nu$; then $\nu'=1$, $\mu'=0$, and we obtain $N_{b,0}$. (b) If $\nu=0$, then $\lambda'=\nu'=0$, $\mu'\alpha_1=\gamma_2\mu$. If $\mu=0$ we have $C_{0,0}$; if $\mu\neq0$, choosing $\alpha_1=\gamma_2\mu$ we have $\mu'=1$, which gives $N_{a^2,0}$.
\end{proof}

\begin{proposition}\label{p2}
If $A^+\cong C_{0,a^2}$, then $A$ is isomorphic to $N_{b,a^2}$, $N_{\alpha a^2,a^2}$ for some $\alpha\in\C$\add{, or $N_{a+ib,a^2}$}.
\end{proposition}
\begin{proof}
We follow the ideas from the previous statement for $b^2=a^2$. Let $ab=-ba=\lambda a+\mu a^2+\nu b$, $b^2=a^2$, and let $a'b'=-b'a'=\lambda'a'+\mu'a'^2+\nu'b'$, $b'^2=a'^2$ in another basis $\{a',a'^2,b'\}$ of the same type; let $\xi(a')=\alpha_1a+\beta_1a^2+\gamma_1b$, $\xi(b')=\alpha_2a+\beta_2a^2+\gamma_2b$. Then
\[
\xi(a'^2)=(\alpha_1^2+\gamma_1^2)a^2,\qquad \xi(b'^2)=(\alpha_2^2+\gamma_2^2)a^2 .
\]
Since $b'^2=a'^2\neq0$, we have $\alpha_2^2+\gamma_2^2=\alpha_1^2+\gamma_1^2=:s\neq0$. The condition $a'b'=-b'a'$ gives $\alpha_1\alpha_2+\gamma_1\gamma_2=0$. Put $\Delta=\alpha_1\gamma_2-\alpha_2\gamma_1$; then $\Delta^2=s^2$, so $\Delta\neq0$, and
\[
\xi(a')\xi(b')=\Delta\,ab=\Delta(\lambda a+\mu a^2+\nu b),\quad
\xi(a'b')=\lambda'\xi(a')+\mu's\,a^2+\nu'\xi(b').
\]
Comparing the coefficients of $a$, $b$, $a^2$, we obtain
\[
\Delta\lambda=\lambda'\alpha_1+\nu'\alpha_2,\quad \Delta\nu=\lambda'\gamma_1+\nu'\gamma_2,\quad
\Delta\mu=\lambda'\beta_1+\nu'\beta_2+\mu' s,
\]
that is
\begin{equation}\label{star}
\lambda'=\gamma_2\lambda-\alpha_2\nu,\quad \nu'=\alpha_1\nu-\gamma_1\lambda,\quad
\mu's=\Delta\mu-\lambda'\beta_1-\nu'\beta_2 .
\end{equation}
Here the matrix $M=\begin{pmatrix}\alpha_1&\gamma_1\\ \alpha_2&\gamma_2\end{pmatrix}$ satisfies $MM^{T}=sE$, $\det M=\Delta=\pm s$, and $\beta_1,\beta_2$ are arbitrary.

(1) If $\lambda=\nu=0$, then $\lambda'=\nu'=0$, and $\mu'$ is determined up to the sign. For $\mu=0$ we have the commutative algebra $C_{0,a^2}$, for $\mu\neq0$ the non-commutative algebras $N_{\mu a^2,a^2}$. We can join these cases in one family $N_{\alpha a^2,a^2}$.

(2) Let $(\lambda,\nu)\neq(0,0)$. Choose $(\alpha_2,\gamma_2)=\tau(\lambda,\nu)$, $\tau\neq0$, so that $\lambda'=0$; then $s=\tau^2(\lambda^2+\nu^2)$ and
\[ \text{the choice requires } s\neq0 \text{, i.e. } \lambda^2+\nu^2\neq0 . \]
\add{(2a) If $\lambda^2+\nu^2\neq0$, take $(\alpha_1,\gamma_1)=\varepsilon\tau(-\nu,\lambda)$, $\varepsilon=\pm1$ (so that $MM^T=sE$). Then $\lambda'=0$, $\nu'=-\varepsilon\tau(\lambda^2+\nu^2)\neq0$, and choosing $\beta_2$ we get $\mu'=0$. Scaling by $\tau$ gives $\nu'=1$, i.e. the algebra $N_{b,a^2}$ (the case $\lambda=0$, $\nu\neq0$ is included).}

\add{(2b) If $\lambda^2+\nu^2=0$, i.e. $\nu=\pm i\lambda$, $\lambda\neq0$, then for every admissible $M$ (with $MM^T=sE$, $s\ne0$) we have $\lambda'=\lambda\sqrt{s}\,e^{\pm i\theta}$, $\nu'=\pm i\lambda'$ (here $M=\sqrt s\begin{pmatrix}\cos\theta&\sin\theta\\ \mp\sin\theta&\pm\cos\theta\end{pmatrix}$), so the relation $\nu=\pm i\lambda$ is preserved and the case $\lambda'=0$ is impossible. Choosing $\sqrt s\,e^{\pm i\theta}=1/\lambda$ and $\beta_1,\beta_2$ we get $\lambda'=1$, $\nu'=\pm i$, $\mu'=0$. The two signs are exchanged by $M$ with $\det M=-s$. Thus we obtain one more algebra $N_{a+ib,a^2}$: $aa=bb=a^2$, $ab=-ba=a+ib$.}

%\del{(2) If $\lambda\neq0$: $\alpha_2=1$, $\gamma_2=\nu/\lambda$, $\lambda'=0$.}
\end{proof}
\add{In the basis $e_1=a+ib$, $e_2=ia^2$, $e_3=\frac i2(a-ib)$ the algebra $N_{a+ib,a^2}$ has the multiplication table
$e_1e_3=e_1+e_2$, $e_3e_1=-e_1+e_2$ (we denote it by ${\mathcal N}_7$ below).}

\begin{proposition}\label{p3}
Algebras $C_{0,0}$, $N_{a,0}$, $N_{b,0}$, $N_{a^2,0}$, $N_{b,a^2}$, $N_{\alpha a^2,a^2}$\add{, $N_{a+ib,a^2}$} are non-isomorphic, except $N_{\alpha a^2,a^2}\cong N_{-\alpha a^2,a^2}$.
\end{proposition}
\begin{proof}
First, commutative algebras are not isomorphic to non-commutative. Second, if $A^+\not\cong B^+$, then $A\not\cong B$. Third, $N_{a,0}^2=\langle a,a^2\rangle$, $N_{b,0}^2=\langle a^2,b\rangle$, $N_{a^2,0}^2=\langle a^2\rangle$. Hence $N_{a^2,0}\not\cong N_{a,0}$ and $N_{a^2,0}\not\cong N_{b,0}$. Since $(N_{a,0}^2)^2\neq0$ and $(N_{b,0}^2)^2=0$, we have $N_{a,0}\not\cong N_{b,0}$.

\add{Consider now the algebras with $A^+\cong C_{0,a^2}$. For $x\in A$ the operator $L_x$ of the left multiplication has the trace $\tr L_x$. For $N_{\alpha a^2,a^2}$ we have $\tr L_x\equiv0$. For $N_{b,a^2}$ we have $\tr L_x=\alpha$ for $x=\alpha a+\beta a^2+\gamma b$, and $b\in\ker\tr L$ satisfies $b^2\neq0$. For $N_{a+ib,a^2}$ we have $\tr L_x=\alpha i-\gamma$, hence $\ker\tr L=\langle a+ib,a^2\rangle$, and $x^2=0$ for all $x\in\ker\tr L$. Therefore, these three types of algebras are pairwise non-isomorphic.}

Let us consider two isomorphic algebras $N=N_{\lambda a^2,a^2}$ and $N'=N_{\mu a^2,a^2}$, and let $\xi$ be an isomorphism between them as in the proof of Proposition~\ref{p2}. It follows that $\alpha_2^2+\gamma_2^2=\alpha_1^2+\gamma_1^2\neq0$, $\alpha_1\alpha_2+\gamma_1\gamma_2=0$ and $\lambda(\alpha_1\gamma_2-\alpha_2\gamma_1)=\mu(\alpha_1^2+\gamma_1^2)$. (1) If $\alpha_1=0$, then $\gamma_1\neq0$, $\gamma_2=0$, $\alpha_2=\pm\gamma_1$ and $\mu=\pm\lambda$. (2) If $\alpha_1\neq0$, then $\alpha_2=-\gamma_1\gamma_2/\alpha_1$, which gives $\alpha_1^2=\gamma_2^2$, $\alpha_2^2=\gamma_1^2$; hence $\alpha_1=\pm\gamma_2$, $\alpha_2=\mp\gamma_1$ and $\mu=\pm\lambda$.
\end{proof}

\subsection{Nilalgebras with nilindex 4}
Let $A$ be a complex $3$-dimensional nilalgebra with nilindex $4$, i.e. there is an element $a$ with $a^2\neq0$, at least one of $aa^2$, $a^2a$ is nonzero, and $a^k=0$ for $k>3$. Suppose that $a^2a\neq0$. If $a^2a\in\langle a,a^2\rangle$, then $a^2a=\alpha a+\beta a^2$ and
\[
0=(a^2a)a=\alpha a^2+\beta a^2a=\alpha\beta a+(\alpha+\beta^2)a^2,
\]
i.e. $\alpha=\beta=0$, a contradiction. Hence we can choose the basis $\{a,a^2,a^2a\}$. Let $aa^2=\gamma_1a+\gamma_2a^2+\gamma_3a^2a$; then $0=(aa^2)a=\gamma_1a^2+\gamma_2a^2a$ and $\gamma_1=\gamma_2=0$. It is easy to see that $A$ is nilpotent. The case $a^2a=0$, $aa^2\neq0$ is similar. The classification of complex $3$-dimensional nilpotent algebras is given in \cite{b6}; hence $A$ is isomorphic to one of the algebras
\[
\mathrm N_1:\ e_1e_1=e_2,\ e_2e_1=e_3;\qquad
\mathrm N_2^\alpha:\ e_1e_1=e_2,\ e_1e_2=e_3,\ e_2e_1=\alpha e_3 .
\]

\subsection{Nilalgebras with nilindex 5}
Let now $A$ have an element $a$ with $a^2\neq0$, at least one of the elements $a^2a^2$, $(aa^2)a$, $a(aa^2)$, $(a^2a)a$, $a(a^2a)$ nonzero, and $a^k=0$ for $k>4$.

(1) If $a^2a^2\neq0$, then as in the previous subsection $a^2a^2\notin\langle a,a^2\rangle$; we choose the basis $\{a,a^2,a^2a^2\}$, and $A$ is nilpotent.

(2) If $(aa^2)a\neq0$, then $aa^2\neq0$ and we choose the basis $\{a,a^2,aa^2\}$. Let $(aa^2)a=\alpha a+\beta a^2+\gamma aa^2$; multiplying this equality by $a$ and $a^2$ (on the left and on the right) and using that all words of weight $\geq5$ in $a$ vanish, we obtain $\alpha=\beta=\gamma=0$, a contradiction.

(3) If one of $a(aa^2)$, $(a^2a)a$, $a(a^2a)$ is nonzero, similar arguments show that this case is not realized.

The classification of complex $3$-dimensional nilpotent algebras \cite{b6} gives
\[
\mathbf N_1:\ e_1e_1=e_2,\ e_2e_2=e_3;\qquad
\mathbf N_2:\ e_1e_1=e_2,\ e_2e_1=e_3,\ e_2e_2=e_3 .
\]

\subsection{Nilalgebras with nilindex \texorpdfstring{$k>5$}{k>5}}
Let $a^k\neq0$ for some $k>5$. Then $a^k$ is a word of one of the forms $((aa)(aa))T_{a}^{k_1}\dots T_a^{k_m}$ with $k_1+\dots+k_m+4=k$, $((aa)a)T_a^{k_1}\dots T_a^{k_m}$ or $(a(aa))T_a^{k_1}\dots T_a^{k_m}$ with $k_1+\dots+k_m+3=k$, where $T_x$ is a left or right multiplication by $x$. Following the case (2) of the previous subsection, we can choose a basis $\{a,a^2,Q\}$, where $Q\in\{(aa)(aa),(aa)a,a(aa)\}$, and applying the same arguments, we see that this case is not realized: there are no $3$-dimensional nilalgebras with nilindex $k>5$.

\subsection{The classification theorem}
The classification of $3$-dimensional nilalgebras with nilindex $2$ ($=$ anticommutative algebras) is given in \cite{b7}. Summarizing this with the previous subsections, we obtain the following theorem.

\begin{theorem}\label{t21}
Let $N$ be a complex $3$-dimensional nilalgebra. Then $N$ is isomorphic to an algebra from the following list (only non-zero products are written; in the anticommutative algebras $e_je_i=-e_ie_j$ is omitted):
\begin{longtable}{@{}>{$}l<{$}p{0.82\textwidth}@{}}
{\mathfrak g}_1 & $e_2e_3=e_1$;\\
{\mathfrak g}_2 & $e_1e_3=e_1$, $e_2e_3=e_2$;\\
{\mathfrak g}_3^\alpha & $e_1e_3=e_1+e_2$, $e_2e_3=\alpha e_2$;\\
{\mathfrak g}_4 & $e_1e_2=e_3$, $e_1e_3=-e_2$, $e_2e_3=e_1$;\\
{\mathcal A}_1^\alpha & $e_1e_2=e_3$, $e_1e_3=e_1+e_3$, $e_2e_3=\alpha e_2$;\\
{\mathcal A}_2 & $e_1e_2=e_1$, $e_2e_3=e_2$;\\
{\mathcal A}_3 & $e_1e_2=e_3$, $e_1e_3=e_1$, $e_2e_3=e_2$;\\
{\mathcal N}_1 & $e_1e_1=e_2$;\\
{\mathcal N}_2 & $e_1e_1=e_2$, $e_1e_3=e_1$, $e_3e_1=-e_1$;\\
{\mathcal N}_3 & $e_1e_1=e_2$, $e_1e_3=e_3$, $e_3e_1=-e_3$;\\
{\mathcal N}_4 & $e_1e_1=e_2$, $e_1e_3=e_2$, $e_3e_1=-e_2$;\\
{\mathcal N}_5 & $e_1e_1=e_2$, $e_1e_3=e_3$, $e_3e_1=-e_3$, $e_3e_3=e_2$;\\
{\mathcal N}_6^\alpha & $e_1e_1=e_2$, $e_1e_3=\alpha e_2$, $e_3e_1=-\alpha e_2$, $e_3e_3=e_2$;\\
\add{{\mathcal N}_7} & \add{$e_1e_3=e_1+e_2$, $e_3e_1=-e_1+e_2$;}\\
\mathrm N_1 & $e_1e_1=e_2$, $e_2e_1=e_3$;\\
\mathrm N_2^\alpha & $e_1e_1=e_2$, $e_1e_2=e_3$, $e_2e_1=\alpha e_3$;\\
\mathbf N_1 & $e_1e_1=e_2$, $e_2e_2=e_3$;\\
\mathbf N_2 & $e_1e_1=e_2$, $e_2e_1=e_3$, $e_2e_2=e_3$.
\end{longtable}
All algebras are non-isomorphic, except ${\mathfrak g}_3^\alpha\cong{\mathfrak g}_3^{\alpha^{-1}}$, ${\mathcal A}_1^\alpha\cong{\mathcal A}_1^{\alpha^{-1}}$ and ${\mathcal N}_6^\alpha\cong{\mathcal N}_6^{-\alpha}$.
\end{theorem}

\add{The algebra ${\mathcal N}_7$ is isomorphic to $N_{a+ib,a^2}$ from Proposition~\ref{p2}; it is the only algebra missing in the published list. It satisfies $\dim\Der{\mathcal N}_7=2$, $\dim\Ann{\mathcal N}_7=1$, $\dim{\mathcal N}_7^2=2$, it is not commutative and ${\mathcal N}_7^+\cong C_{0,a^2}$.}

Let us recall Albert's problem: is every finite-dimensional (commutative) power associative nilalgebra solvable? Correa and Hentzel \cite{b3} constructed, for each $n>3$, a non-(anti)commutative $n$-dimensional non-solvable nilalgebra. If a $2$-dimensional algebra is a nilalgebra, then it is commutative or anticommutative. Hence, Theorem~\ref{t21} gives the following corollary.

\begin{corollary}
Albert's problem has a positive answer for all non-anticommutative $3$-dimensional algebras.
\end{corollary}

\begin{remark}
The Nagata--Higman--Dubnov--Ivanov theorem says that every associative nilalgebra is nilpotent. The algebra ${\mathcal N}_3$ is a non-nilpotent nilalgebra satisfying all identities of the following type:
\begin{multline*}
\alpha_1(xy)z+\alpha_2(yx)z+\alpha_3(xz)y+\alpha_4(zy)x+\alpha_5(yz)x+\alpha_6(zx)y\\
+\alpha_7z(xy)+\alpha_8z(yx)+\alpha_9y(xz)+\alpha_{10}x(zy)\\
+(-\alpha_1+\alpha_2+\alpha_7-\alpha_8-\alpha_4+\alpha_5+\alpha_{10})x(yz)\\
+(-\alpha_1+\alpha_2+\alpha_7-\alpha_8-\alpha_3+\alpha_6+\alpha_9)y(zx)=0.
\end{multline*}
Hence none of the varieties defined by such an identity has an analog of this theorem. In particular, the following identities have this type: (1) (right) Leibniz: $(xy)z=(xz)y+x(yz)$; (2) reverse (right) Leibniz: $(xy)z=(zy)x+y(zx)$; (3) weakly associative: $(xy)z-x(yz)+(yz)x-y(zx)-(yx)z+y(xz)=0$; (4) $2$-step Jordan nilpotent: $(x\cdot y)\cdot z=(xy)z+(yx)z+z(xy)+z(yx)=0$; (5) almost anticommutative: $(xy)z+(yx)z=0$.
\end{remark}

By \cite{b1}, the intersection of right mono Leibniz algebras (i.e. algebras in which every one-generated subalgebra is a right Leibniz algebra) and left mono Leibniz algebras is the variety of nilalgebras with nilindex $3$. Hence we have the following corollary.

\begin{corollary}
The algebraic classification of symmetric mono Leibniz algebras is given in Theorem~\ref{t21}; namely, it consists of the algebras of types ${\mathfrak g}_i$, ${\mathcal A}_i$ or ${\mathcal N}_i$.
\end{corollary}

\section{Degenerations of $3$-dimensional nilalgebras}
The study of varieties of non-associative algebras from a geometric point of view has a long history (see \cite{b5,b6,b7,b8,b9,b10,b12} and references therein). The geometric classification of algebras of a given variety is based on the notion of degeneration, which is a ``dual'' notion to deformation \cite{b2,b13}.

\subsection{Definitions and notation}
Given an $n$-dimensional vector space $\mathbb V$, the set $\mathrm{Hom}(\mathbb V\otimes\mathbb V,\mathbb V)\cong\mathbb V^*\otimes\mathbb V^*\otimes\mathbb V$ is a vector space of dimension $n^3$. This space has the structure of the affine variety $\C^{n^3}$. Indeed, fix a basis $e_1,\dots,e_n$ of $\mathbb V$. Then any $\mu\in\mathrm{Hom}(\mathbb V\otimes\mathbb V,\mathbb V)$ is determined by $n^3$ structure constants $c_{ij}^k\in\C$ such that $\mu(e_i\otimes e_j)=\sum_{k=1}^nc_{ij}^ke_k$. A subset of $\mathrm{Hom}(\mathbb V\otimes\mathbb V,\mathbb V)$ is Zariski-closed if it can be defined by a set of polynomial equations in the variables $c_{ij}^k$ ($1\le i,j,k\le n$).

Let $T$ be a set of polynomial identities. The set of algebra structures on $\mathbb V$ satisfying the identities from $T$ is a Zariski-closed subset of $\mathrm{Hom}(\mathbb V\otimes\mathbb V,\mathbb V)$; denote it by $\mathbb L(T)$. The general linear group $\mathrm{GL}(\mathbb V)$ acts on $\mathbb L(T)$ by conjugation
\[
(g*\mu)(x\otimes y)=g\mu(g^{-1}x\otimes g^{-1}y),\qquad x,y\in\mathbb V,\ \mu\in\mathbb L(T),\ g\in\mathrm{GL}(\mathbb V).
\]
Thus $\mathbb L(T)$ is decomposed into $\mathrm{GL}(\mathbb V)$-orbits corresponding to the isomorphism classes of algebras. Let $\mathcal O(\mu)$ denote the orbit of $\mu\in\mathbb L(T)$ and $\overline{\mathcal O(\mu)}$ its Zariski closure.

Let $\mathbf A$ and $\mathbf B$ be two $n$-dimensional algebras satisfying the identities from $T$, represented by $\mu,\lambda\in\mathbb L(T)$. We say that $\mathbf A$ degenerates to $\mathbf B$ and write $\mathbf A\to\mathbf B$ if $\lambda\in\overline{\mathcal O(\mu)}$. In this case $\overline{\mathcal O(\lambda)}\subset\overline{\mathcal O(\mu)}$, so the definition does not depend on the choice of $\mu$ and $\lambda$. If $\mathbf A\not\cong\mathbf B$, the degeneration is called proper. We write $\mathbf A\not\to\mathbf B$ if $\lambda\notin\overline{\mathcal O(\mu)}$.

Let $\mathbf A$ be represented by $\mu\in\mathbb L(T)$. Then $\mathbf A$ is rigid in $\mathbb L(T)$ if $\mathcal O(\mu)$ is open in $\mathbb L(T)$. A subset of a variety is irreducible if it is not a union of two proper closed subsets; a maximal irreducible closed subset is an irreducible component. Every affine variety is a finite union of its irreducible components in a unique way, and $\mathbf A$ is rigid if and only if $\overline{\mathcal O(\mu)}$ is an irreducible component of $\mathbb L(T)$.

\subsection{Method of the description of degenerations}
We use the methods applied to Lie algebras in \cite{b5}. If $\mathbf A\to\mathbf B$ and $\mathbf A\not\cong\mathbf B$, then $\dim\Der\mathbf A<\dim\Der\mathbf B$, where $\Der\mathbf A$ is the algebra of derivations of $\mathbf A$; equivalently, $\dim\mathcal O(\mu)>\dim\mathcal O(\lambda)$. We compute the dimensions of the algebras of derivations and check $\mathbf A\to\mathbf B$ only for pairs with $\dim\Der\mathbf A<\dim\Der\mathbf B$.

To prove a degeneration we construct families of matrices parameterized by $t$. Let $e_1,\dots,e_n$ be a basis of $\mathbb V$ and $c_{ij}^k$ the structure constants of $\lambda$ in this basis. If there exist $a_i^j(t)\in\C$ such that $E_i^t=\sum_ja_i^j(t)e_j$ form a basis of $\mathbb V$ for every $t\in\C^*$, and the structure constants $c_{ij}^k(t)$ of $\mu$ in the basis $E_1^t,\dots,E_n^t$ are polynomials in $t$ (or rational functions without a pole at $0$) with $c_{ij}^k(0)=c_{ij}^k$, then $\mathbf A\to\mathbf B$; $E_1^t,\dots,E_n^t$ is called a parameterized basis for $\mathbf A\to\mathbf B$. For a series of algebras $\mathbf A(*)=\{\mathbf A(\alpha)\}_{\alpha\in I}$ one similarly constructs $E_i^t$ and a parameterized index $f(t)\in I$ such that $\mu(f(t))$ in the basis $E^t_i$ tends to $\lambda$.

To prove non-degenerations we use the following facts. If $\dim\Der\mathbf A(\alpha)>\dim\Der\mathbf B$ for all $\alpha\in I$, then $\mathbf A(*)\not\to\mathbf B$. We also use the following lemma, whose proof is the same as that of \cite[Lemma 1.5]{b5}.

\begin{lemma}[\cite{b5}]\label{l31}
Let $\mathfrak B$ be a Borel subgroup of $\mathrm{GL}(\mathbb V)$ and $\mathcal R\subset\mathbb L(T)$ a $\mathfrak B$-stable closed subset. If $\mathbf A(*)\to\mathbf B$ and for any $\alpha\in I$ the algebra $\mathbf A(\alpha)$ can be represented by a structure $\mu(\alpha)\in\mathcal R$, then there is $\lambda\in\mathcal R$ representing $\mathbf B$.
\end{lemma}

\add{In the proof of Theorem~\ref{t31} we shall use the following simple invariants. For $x\in\mathbb V$ denote by $L_x$ the operator of left multiplication, and for a linear form $\varphi$ consider the bilinear form $\varphi(xy+yx)$.}

\begin{lemma}\label{l32}
\add{Let $\mathbf A\to\mathbf B$. Then:}
\begin{enumerate}
\item[\add{(a)}] \add{$\dim\Ann\mathbf A\le\dim\Ann\mathbf B$ and $\dim\mathbf A^2\ge\dim\mathbf B^2$;}
\item[\add{(b)}] \add{if $\mathbf A$ is anticommutative (respectively, nilpotent, or satisfies a polynomial identity), then so does $\mathbf B$;}
\item[\add{(c)}] \add{if $\sigma_2(L_x)=0$ for all $x\in\mathbf A$, where $\sigma_2$ is the second elementary symmetric function of the eigenvalues, then $\sigma_2(L_x)=0$ for all $x\in\mathbf B$;}
\item[\add{(d)}] \add{if the rank of the bilinear form $\varphi(xy+yx)$ is at most $r$ for all $\varphi\in\mathbb V^*$ in $\mathbf A$, then the same holds in $\mathbf B$.}
\end{enumerate}
\end{lemma}
\begin{proof}
\add{Each of the conditions in (b)--(d) is given by polynomial equations in the structure constants and is invariant under the change of basis, so it defines a closed $\mathrm{GL}(\mathbb V)$-stable subset of $\mathbb L(T)$, which contains the closure of the orbit of every algebra satisfying it. In (a) the numbers $\dim\mathbf A^2$ and $3-\dim\Ann\mathbf A$ are the ranks of linear maps depending polynomially on the structure constants, hence they can only drop in the closure.}
\end{proof}

\add{The values of these invariants are as follows. The condition (c) holds in ${\mathfrak g}_1$, ${\mathfrak g}_3^0$, in all algebras ${\mathcal N}_i$ ($1\le i\le7$) and in all nilpotent algebras, and fails in ${\mathfrak g}_2$, ${\mathfrak g}_3^\alpha$ ($\alpha\ne0$), ${\mathfrak g}_4$, ${\mathcal A}_1^\alpha$, ${\mathcal A}_2$, ${\mathcal A}_3$ (for example, $\sigma_2(L_x)=\alpha x_3^2$ for ${\mathfrak g}_3^\alpha$ and $x_3^2$ for ${\mathfrak g}_2$). The form $\varphi(xy+yx)$ has rank at most $1$ in ${\mathcal N}_2$, ${\mathcal N}_3$, and rank $2$ in ${\mathcal N}_5$, ${\mathcal N}_6^\beta$, ${\mathcal N}_7$.}

\subsection{Degenerations of $3$-dimensional nilalgebras}
\begin{theorem}\label{t31}
The graph of degenerations of the variety of $3$-dimensional nilalgebras is presented in Figure below, where the number on the left is the dimension of the orbit of the algebras on the same level. In particular, this variety has dimension $9$, two rigid algebras ${\mathcal N}_5$ and $\mathbf N_2$, and three irreducible components, described below:
\begin{align*}
\overline{\mathcal O({\mathcal A}_1^\alpha)}&=\{{\mathcal A}_1^\alpha,{\mathcal A}_2,{\mathcal A}_3,{\mathfrak g}_1,{\mathfrak g}_2,{\mathfrak g}_3^\alpha,{\mathfrak g}_4,\C^3\},\\
\overline{\mathcal O({\mathcal N}_5)}&=\{{\mathcal N}_1,{\mathcal N}_2,{\mathcal N}_3,{\mathcal N}_4,{\mathcal N}_5,{\mathcal N}_6^\alpha,\add{{\mathcal N}_7},{\mathfrak g}_1,{\mathfrak g}_3^0,\C^3\},\\
\overline{\mathcal O(\mathbf N_2)}&=\{\mathbf N_1,\mathbf N_2,\mathrm N_1,\mathrm N_2^\alpha,{\mathcal N}_4,{\mathcal N}_6^\alpha,{\mathfrak g}_1,\C^3\}.
\end{align*}
In particular, the variety of $3$-dimensional nilalgebras with nilindex $3$ has dimension $9$, one rigid algebra ${\mathcal N}_5$ and two irreducible components, defined by ${\mathcal A}_1^\alpha$ and ${\mathcal N}_5$; the variety of $3$-dimensional nilalgebras with nilindex $4$ has dimension $9$, one rigid algebra ${\mathcal N}_5$ and three irreducible components, defined by ${\mathcal A}_1^\alpha$, $\mathrm N_2^\alpha$ and ${\mathcal N}_5$.

\end{theorem}

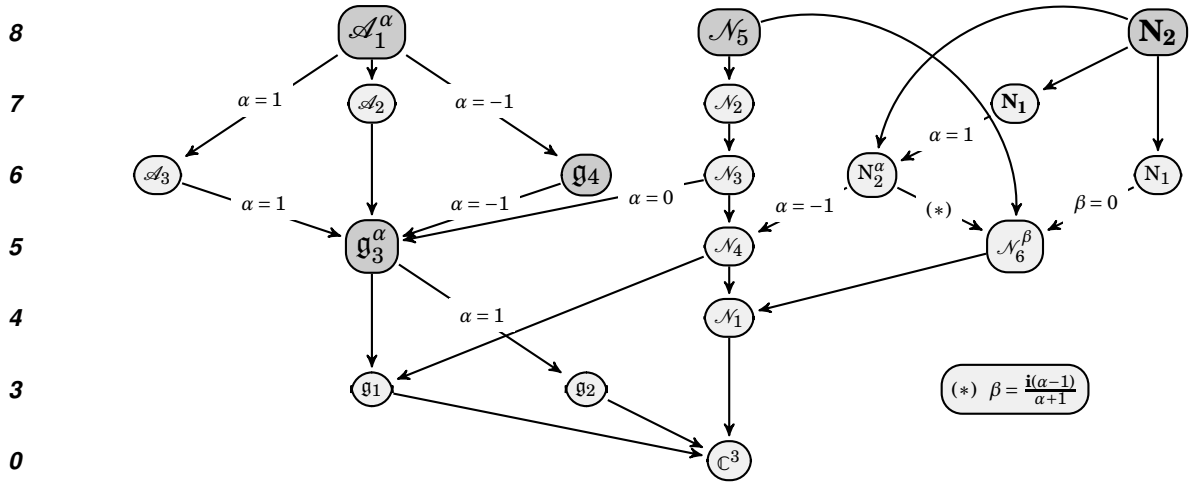
\begin{center}

\

\

\begin{tikzpicture}[->,>=stealth',shorten >=0.05cm, auto, node distance=0.95cm, thick,
                    main  node/.style={rectangle,draw,fill=gray!12, rounded corners=1.5ex,font=\sffamily \tiny \bfseries },
                    rigid node/.style={rectangle,draw,fill=black!20,rounded corners=1.5ex,font=\sffamily \bf \bfseries }, 
                    ver  node/.style={rectangle,draw,fill=gray!12, rounded corners=1.5ex,font=\sffamily \bf \bfseries },
                    ser node/.style={rectangle,draw,fill=green!20,rounded corners=1.5ex,font=\sffamily \bf \bfseries },
                    nrig node/.style={rectangle,draw,fill=cyan!20,rounded corners=1.5ex,font=\sffamily \tiny \bfseries },
                    nrigdead node/.style={rectangle,draw,fill=teal!60,rounded corners=1.5ex,font=\sffamily \tiny \bfseries }, 
                    lie node/.style={rectangle,draw,fill=yellow!20,rounded corners=1.5ex,font=\sffamily \tiny \bfseries },
                    lier node/.style={rectangle,draw,fill=yellow!80,rounded corners=1.5ex,font=\sffamily \tiny \bfseries },
                    style={draw,font=\sffamily \scriptsize \bfseries }]
\node (0)   {};

\node (00a1) [right of=0] {};
\node (00a2) [right of=00a1] {};
\node (00a3) [right of=00a2] {};
\node (00a4) [right of=00a3] {};
\node (00a5) [right of=00a4] {};
\node (00a6) [right of=00a5] {};
\node (00a7) [right of=00a6] {};
\node (00a8) [right of=00a7] {};
\node (00a9) [right of=00a8] {};
\node (00a10) [right of=00a9] {};
\node (00a11) [right of=00a10] {};
\node (00a12) [right of=00a11] {};
\node (00a13) [right of=00a12] {};
\node (00a14) [right of=00a13] {};
\node (00a15) [right of=00a14] {};
\node (00a16) [right of=00a15] {};
\node (00a17) [right of=00a16] {};
\node (00a18) [right of=00a17] {};
\node (00a19) [right of=00a18] {};
\node (00a20) [right of=00a19] {};
\node (00a21) [right of=00a20] {};
\node (00a22) [right of=00a21] {};
\node (00a23) [right of=00a22] {};
\node (00a24) [right of=00a23] {};
\node (00a25) [right of=00a24] {};
\node (00a26) [right of=00a25] {};
\node (00a27) [right of=00a26] {};

\node (01) [below of=0] {8};

\node (01a1) [right of=01] {};
\node (01a2) [right of=01a1] {};
\node (01a3) [right of=01a2] {};
\node (01a4) [right of=01a3] {};
\node (01a5) [right of=01a4] {};
\node (01a6) [right of=01a5] {};
\node (01a7) [right of=01a6] {};
\node (01a8) [right of=01a7] {};
\node (01a9) [right of=01a8] {};
\node (01a10) [right of=01a9] {};
\node (01a11) [right of=01a10] {};
\node (01a12) [right of=01a11] {};
\node (01a13) [right of=01a12] {};
\node (01a14) [right of=01a13] {};
\node (01a15) [right of=01a14] {};
\node (01a16) [right of=01a15] {};
\node (01a17) [right of=01a16] {};

\node (02)[below of=01]  {7};

\node (02a1) [right of=02] {};
\node (02a2) [right of=02a1] {};
\node (02a3) [right of=02a2] {};
\node (02a4) [right of=02a3] {};
\node (02a5) [right of=02a4] {};
\node (02a6) [right of=02a5] {};
\node (02a7) [right of=02a6] {};
\node (02a8) [right of=02a7] {};
\node (02a9) [right of=02a8] {};
\node (02a10) [right of=02a9] {};
\node (02a11) [right of=02a10] {};
\node (02a12) [right of=02a11] {};
\node (02a13) [right of=02a12] {};
\node (02a14) [right of=02a13] {};
\node (02a15) [right of=02a14] {};
\node (02a16) [right of=02a15] {};
\node (02a17) [right of=02a16] {};
\node (02a18) [right of=02a17] {};
\node (02a19) [right of=02a18] {};
\node (02a20) [right of=02a19] {};
\node (02a21) [right of=02a20] {};
\node (02a22) [right of=02a21] {};
\node (02a23) [right of=02a22] {};
\node (02a24) [right of=02a23] {};
\node (02a25) [right of=02a24] {};
\node (02a26) [right of=02a25] {};
\node (02a27) [right of=02a26] {};

\node (03)[below of=02]{6};

\node (03a1) [right of=03] {};
\node (03a2) [right of=03a1] {};
\node (03a3) [right of=03a2] {};
\node (03a4) [right of=03a3] {};
\node (03a5) [right of=03a4] {};
\node (03a6) [right of=03a5] {};
\node (03a7) [right of=03a6] {};
\node (03a8) [right of=03a7] {};
\node (03a9) [right of=03a8] {};
\node (03a10) [right of=03a9] {};
\node (03a11) [right of=03a10] {};
\node (03a12) [right of=03a11] {};
\node (03a13) [right of=03a12] {};
\node (03a14) [right of=03a13] {};
\node (03a15) [right of=03a14] {};
\node (03a16) [right of=03a15] {};
\node (03a17) [right of=03a16] {};
\node (03a18) [right of=03a17] {};
\node (03a19) [right of=03a18] {};
\node (03a20) [right of=03a19] {};
\node (03a21) [right of=03a20] {};
\node (03a22) [right of=03a21] {};
\node (03a23) [right of=03a22] {};
\node (03a24) [right of=03a23] {};
\node (03a25) [right of=03a24] {};
\node (03a26) [right of=03a25] {};
\node (03a27) [right of=03a26] {};
*\node (03a28) [right of=03a27] {};

\node (04) [below of=03] {5};

\node (04a1) [right of=04] {};
\node (04a2) [right of=04a1] {};
\node (04a3) [right of=04a2] {};
\node (04a4) [right of=04a3] {};
\node (04a5) [right of=04a4] {};
\node (04a6) [right of=04a5] {};+
\node (04a7) [right of=04a6] {};
\node (04a8) [right of=04a7] {};
\node (04a9) [right of=04a8] {};
\node (04a10) [right of=04a9] {};
\node (04a11) [right of=04a10] {};
\node (04a12) [right of=04a11] {};
\node (04a13) [right of=04a12] {};
\node (04a14) [right of=04a13] {};
\node (04a15) [right of=04a14] {};
\node (04a16) [right of=04a15] {};
\node (04a17) [right of=04a16] {};
\node (04a18) [right of=04a17] {};
\node (04a19) [right of=04a18] {};
\node (04a20) [right of=04a19] {};
\node (04a21) [right of=04a20] {};
\node (04a22) [right of=04a21] {};
\node (04a23) [right of=04a22] {};
\node (04a24) [right of=04a23] {};
\node (04a25) [right of=04a24] {};
\node (04a26) [right of=04a25] {};
\node (04a27) [right of=04a26] {};
\node (04a28) [right of=04a27] {};

\node (051)[below of=04]{};
\node (05) [below of=04] {4};

\node (05a1) [right of=05] {};
\node (05a2) [right of=05a1] {};
\node (05a3) [right of=05a2] {};
\node (05a4) [right of=05a3] {};
\node (05a5) [right of=05a4] {};
\node (05a6) [right of=05a5] {};
\node (05a7) [right of=05a6] {};
\node (05a8) [right of=05a7] {};
\node (05a9) [right of=05a8] {};
\node (05a10) [right of=05a9] {};
\node (05a11) [right of=05a10] {};
\node (05a12) [right of=05a11] {};
\node (05a13) [right of=05a12] {};
\node (05a14) [right of=05a13] {};
\node (05a15) [right of=05a14] {};
\node (05a16) [right of=05a15] {};
\node (05a17) [right of=05a16] {};
\node (05a18) [right of=05a17] {};
\node (05a19) [right of=05a18] {};
\node (05a20) [right of=05a19] {};
\node (05a21) [right of=05a20] {};
\node (05a22) [right of=05a21] {};
\node (05a23) [right of=05a22] {};
\node (05a24) [right of=05a23] {};
\node (05a25) [right of=05a24] {};
\node (05a26) [right of=05a25] {};
\node (05a27) [right of=05a26] {};
\node (05a28) [right of=05a27] {};

\node (061)[below of=05]{};
\node (06)[below of=05] {3};

\node (06a1) [right of=06] {};
\node (06a2) [right of=06a1] {};
\node (06a3) [right of=06a2] {};
\node (06a4) [right of=06a3] {};
\node (06a5) [right of=06a4] {};
\node (06a6) [right of=06a5] {};
\node (06a7) [right of=06a6] {};
\node (06a8) [right of=06a7] {};
\node (06a9) [right of=06a8] {};
\node (06a10) [right of=06a9] {};
\node (06a11) [right of=06a10] {};
\node (06a12) [right of=06a11] {};
\node (06a13) [right of=06a12] {};
\node (06a14) [right of=06a13] {};
\node (06a15) [right of=06a14] {};
\node (06a16) [right of=06a15] {};
\node (06a17) [right of=06a16] {};
\node (06a18) [right of=06a17] {};
\node (06a19) [right of=06a18] {};
\node (06a20) [right of=06a19] {};
\node (06a21) [right of=06a20] {};
\node (06a22) [right of=06a21] {};
\node (06a23) [right of=06a22] {};
\node (06a24) [right of=06a23] {};
\node (06a25) [right of=06a24] {};
\node (06a26) [right of=06a25] {};
\node (06a27) [right of=06a26] {};
\node (06a28) [right of=06a27] {};

\node (071)[below of=06]{0};

\node (07a1) [right of=06] {};
\node (07a2) [right of=071] {};
\node (07a3) [right of=07a2] {};
\node (07a4) [right of=07a3] {};
\node (07a5) [right of=07a4] {};
\node (07a6) [right of=07a5] {};
\node (07a7) [right of=07a6] {};
\node (07a8) [right of=07a7] {};
\node (07a9) [right of=07a8] {};
\node (07a10) [right of=07a9] {};
\node (07a11) [right of=07a10] {};
\node (07a12) [right of=07a11] {};
\node (07a13) [right of=07a12] {};
\node (07a14) [right of=07a13] {};
\node (07a15) [right of=07a14] {};
\node (07a16) [right of=07a15] {};
\node (07a17) [right of=07a16] {};
\node (07a18) [right of=07a17] {};
\node (07a19) [right of=07a18] {};
\node (07a20) [right of=07a19] {};
\node (07a21) [right of=07a20] {};
\node (07a22) [right of=07a21] {};
\node (07a23) [right of=07a22] {};
\node (07a24) [right of=07a23] {};
\node (07a25) [right of=07a24] {};
\node (07a26) [right of=07a25] {};
\node (07a27) [right of=07a26] {};
\node (07a28) [right of=07a27] {};

\node   [rigid node] (ant5)     [right of =01a4] {$\mathcal{A}_1^{\alpha}$};
\node   [main node] (ant6)     [right of =02a4] {$\mathcal{A}_2$};
\node   [main node] (ant8)     [right of =03a1] {$\mathcal{A}_3$};

\node   [main node] (g1)     [right of =06a4] {$\mathfrak{g}_{1}$};
\node   [main node] (g32)     [right of =06a7] {$\mathfrak{g}_{2}$};

\node   [rigid node] (g31)     [right of =04a4] {$\mathfrak{g}^{\alpha}_{3}$};

\node   [rigid node] (g4)     [right of =03a7] {$\mathfrak{g}_{4}$};

\node   [main node] (n1)     [right of =05a9] {$\mathcal{N}_1$};

\node   [main node] (n4)     [above of =n1] {$\mathcal{N}_4$};

\node   [main node] (n6)     [left of =04a15] {$\mathcal{N}_6^{\beta}$};

\node   [main node] (n3)     [above of =n4] {$\mathcal{N}_3$};

\node     (n00)     [above of =n3] { };
\node    [main node] (n2)     [left of =n00] {$\mathcal{N}_2$};
\node   [main node] (n7)     [right of =n00] {$\mathcal{N}_7$};

\node    [rigid node](n5)     [above of =n00] {$\mathcal{N}_5$};

\node    [rigid node](n2bf)     [right of =01a15] {$\bf{N}_2$};

\node    [main node](n1bf)     [right of =02a13] {$\bf{N}_1$};

\node    [main node](n1rm)     [right of =03a15] {$\rm{N}_1$};

\node    [main node](n2rm)     [right of =03a11] {$\rm{N}_2^\alpha$};

\node    [main node](xxx)     [right of =06a13] {$(*) \ \ \beta=\pm\frac{{\bf i}(\alpha-1)}{\alpha+1}$};

\node  [main node] (t20) [right of=07a10] {$\mathbb{C}^3$};

\path[every node/.style={font=\sffamily\small}]

(n5) edge  [bend left=55]  node[above]{ } (n6)
(n2bf) edge   [bend left=-50]  node[above] {} (n2rm)
(n5) edge     node[above] {} (n7)
(n7) edge     node[above] {} (n3)
(n7) edge   [bend left= 15 ]  node[above] {} (n6)

(n2bf) edge     node[above] {} (n1bf)
(n2bf) edge     node[above] {} (n1rm)

(n1bf) edge   node[above=2,   right=-10, fill=white]{\tiny $\alpha=1$} (n2rm)

(n1rm) edge   node[above=2,   right=-10, fill=white]{\tiny $\beta=0$} (n6)

(n2rm) edge   node[above=2,   right=-14, fill=white]{\tiny $\alpha=-1$} (n4)

(n2rm) edge   node[above=0,   right=-10, fill=white]{\tiny $(*)$} (n6)

(n1) edge     node[above] {} (t20)
(n4) edge     node[above] {} (n1)
(n4) edge     node[above] {} (g1)

(n6) edge     node[above] {} (n1)

(n3) edge     node[above] {} (n4)

(n2) edge [bend left=0]    node[above] {} (n3)

(n5) edge     node[above] {} (n2)

(g1) edge        node[above] {} (t20)
(g32) edge     node[above] {} (t20)

(g31) edge  node[above] {} (g1)
(g31) edge   node[above=-5, fill=white]{\tiny $\alpha=1$} (g32)
(g4) edge   node[above=-4, fill=white]{\tiny $\alpha=-1$} (g31)

(ant5) edge   node[above=-4, fill=white]{\tiny $\alpha=1$} (ant8)
(ant5) edge   node{} (ant6)
(ant5) edge  node[above=2, right=-15, fill=white]{\tiny $\alpha=-1$} (g4)

(ant6) edge [bend left=0]  node{} (g31)
(n3) edge   node[above=5,   right=25, fill=white]{\tiny $\alpha=0$} (g31)

(ant8) edge   node[above=-4, fill=white]{\tiny $\alpha=1$} (g31);

\end{tikzpicture}

\end{center}

\begin{proof}
The subgraph of degenerations between anticommutative algebras ${\mathcal A}_i$, ${\mathfrak g}_i$ is given in \cite{b7}. The subgraph of degenerations between nilpotent algebras, namely ${\mathfrak g}_1$, ${\mathcal N}_1$, ${\mathcal N}_4$, ${\mathcal N}_6^\beta$, $\mathrm N_1$, $\mathrm N_2^\alpha$, $\mathbf N_1$, $\mathbf N_2$, is given in \cite{b6}; in particular, $\mathbf N_2\to\mathbf N_1$, $\mathbf N_2\to\mathrm N_2^\alpha$, $\mathrm N_2^\alpha\to\mathrm N_1$, $\mathrm N_2^\alpha\to{\mathcal N}_6^\beta$ if and only if $\beta=\pm i\frac{\alpha-1}{\alpha+1}$ (for $\alpha\neq-1$), $\mathrm N_2^{-1}\to{\mathcal N}_4$, $\mathbf N_1\to{\mathcal N}_6^0$, $\mathrm N_1\to{\mathcal N}_1$, ${\mathcal N}_4\to{\mathcal N}_1,{\mathfrak g}_1$, ${\mathcal N}_6^\beta\to{\mathcal N}_1$. We complete these subgraphs to the full graph.

The list of primary degenerations is given below.
\begin{longtable}{ll}
${\mathcal N}_5\to{\mathcal N}_2$ & 
$E_1^t=t^{-1}e_3,\ E_2^t=t^{-2}e_2+te_3,\ E_3^t=-e_1;$\\[3pt]
${{\mathcal N}_5\to{\mathcal N}_7}$ & 
${E_1^t=t(e_1+ie_3),\ E_2^t=t(it-2)e_2,\ E_3^t=\tfrac i2\bigl((2i+t)e_1+(2-it)e_3\bigr);}$\\[3pt]
${\mathcal N}_7\to{\mathcal N}_3$ & 
${E_1^t=-(e_1+e_3),\ E_2^t=2e_2,\ E_3^t=t(e_1-e_2);}$\\[3pt]
${{\mathcal N}_7\to{\mathcal N}_6^\beta}$ & 
${E_1^t=t(e_1-i\beta e_2+e_3),\ E_2^t=2t^2e_2,\ E_3^t=t(ie_1+\beta e_2-ie_3);}$\\[3pt]
${\mathcal N}_2\to{\mathcal N}_3$ & 
$E_1^t=e_1\add{-e_2}-e_3,\ E_2^t=e_2,\ E_3^t=te_3;$\\[3pt]
${\mathcal N}_3\to{\mathcal N}_4$ & 
$E_1^t=te_1,\ E_2^t=t^2e_3,\ E_3^t=-te_2+te_3;$\\[3pt]
${\mathcal N}_3\to{\mathfrak g}_3^0$ & 
$E_1^t=-t^{-1}e_2+e_3,\ E_2^t=t^{-1}e_2,\ E_3^t=\add{-e_1}\mdel{-te_1}.$
\end{longtable}
\add{(The basis $(e_1-e_3,e_3,te_3)$ for ${\mathcal N}_2\to{\mathcal N}_3$ in the published version does not give ${\mathcal N}_3$; the basis for ${\mathcal N}_3\to{\mathfrak g}_3^0$ contained a misprint in the third vector. The published degeneration ${\mathcal N}_5\to{\mathcal N}_6^\beta$ is now obtained as the composition ${\mathcal N}_5\to{\mathcal N}_7\to{\mathcal N}_6^\beta$.)}

The list of primary non-degenerations follows from Lemma~\ref{l32}, the formulas for $\dim\Der$ and the results of \cite{b6,b7}:
\begin{enumerate}
\item an anticommutative algebra degenerates only to anticommutative algebras (Lemma~\ref{l32}(b)); hence ${\mathcal A}_i,{\mathfrak g}_i\not\to{\mathcal N}_j,\mathrm N_j,\mathbf N_j$;
\item \add{by Lemma~\ref{l32}(c), ${\mathcal N}_i\not\to{\mathcal A}_1^\alpha,{\mathcal A}_2,{\mathcal A}_3,{\mathfrak g}_2,{\mathfrak g}_3^\alpha\,(\alpha\ne0),{\mathfrak g}_4$ for $1\le i\le7$; in particular ${\mathcal N}_5\not\to{\mathfrak g}_2,{\mathfrak g}_3^{\alpha\ne0}$, and the only anticommutative algebras in $\overline{\mathcal O({\mathcal N}_5)}$ are ${\mathfrak g}_1,{\mathfrak g}_3^0$ and $\C^3$;}
\item nilpotent algebras degenerate only to nilpotent algebras; hence $\mathrm N_i,\mathbf N_i\not\to{\mathcal N}_2,{\mathcal N}_3,{\mathcal N}_5,{\mathcal N}_7,{\mathfrak g}_3^0$;
\item the algebras ${\mathcal N}_i$ ($i\le7$) satisfy $x^3=0$, which is not satisfied by $\mathrm N_i,\mathbf N_i$; hence ${\mathcal N}_i\not\to\mathrm N_j,\mathbf N_j$;
\item \add{${\mathcal N}_2,{\mathcal N}_3\not\to{\mathcal N}_6^\beta$ for all $\beta$ by Lemma~\ref{l32}(d) (the rank of $\varphi(xy+yx)$ is at most $1$ in ${\mathcal N}_2,{\mathcal N}_3$ and equals $2$ in ${\mathcal N}_6^\beta$);}
\item \add{${\mathcal N}_2\not\to{\mathcal N}_7$, ${\mathcal N}_7\not\to{\mathcal N}_2$, ${\mathcal N}_4\not\to{\mathfrak g}_3^0$, ${\mathcal N}_6^\beta\not\to{\mathcal N}_4,{\mathfrak g}_3^0$ since the algebras have equal $\dim\Der$; ${\mathcal N}_1\not\to{\mathfrak g}_1$, since $\dim\Ann{\mathcal N}_1=2>\dim\Ann{\mathfrak g}_1$ (Lemma~\ref{l32}(a)).}
\end{enumerate}
All remaining non-degenerations follow from these and the transitivity of degeneration. The components are obtained from the graph: the algebras ${\mathcal N}_5$ and $\mathbf N_2$ are rigid since they have maximal orbits among the corresponding classes and by the non-degenerations above nothing else degenerates to them, and $\overline{\mathcal O({\mathcal A}_1^\alpha)}$ has dimension $9$. This completes the proof.
\end{proof}

Let us recall that the variety of $3$-dimensional nilpotent algebras is irreducible and defined by a rigid algebra, and the variety of $n$-dimensional ($n>3$) nilpotent algebras is irreducible but has no rigid algebras \cite{b9}. This observation and Theorem~\ref{t31} give the following question.

\medskip
\noindent\textbf{Open question.} Are there rigid algebras in the variety of $n$-dimensional ($n>3$) nilalgebras with nilindex $k$?

\appendix
\section*{Appendix. Differences from the published version}
\noindent\textbf{Theorem~\ref{t21}:} statement changed. The algebra ${\mathcal N}_7$: $e_1e_3=e_1+e_2$, $e_3e_1=-e_1+e_2$, missing in the published classification, has been added. Accordingly, Proposition~\ref{p2} and its proof (case $\lambda^2+\nu^2=0$ in the classification of algebras with $A^+\cong C_{0,a^2}$), Theorem~\ref{t31}, the list of primary degenerations and non-degenerations and the figure have been corrected and completed.

\noindent\textbf{Theorem~\ref{t31}:} statement changed. The components $\overline{\mathcal O({\mathcal N}_5)}$ now contain ${\mathcal N}_7$; the degenerations ${\mathcal N}_5\to{\mathcal N}_7$, ${\mathcal N}_7\to{\mathcal N}_3$, ${\mathcal N}_7\to{\mathcal N}_6^\beta$ have been constructed. The parameterized bases for ${\mathcal N}_2\to{\mathcal N}_3$ and ${\mathcal N}_3\to{\mathfrak g}_3^0$ have been corrected, and the non-degenerations have been proved by the invariants of Lemma~\ref{l32} instead of the sets $\mathcal R$.


\begin{thebibliography}{99}
\bibitem{b1} S. Benayadi, I. Kaygorodov, and F. Mhamdi, \emph{Symmetric Zinbiel superalgebras}, Communications in Algebra, 51 (2023), no. 1, 224--238; \doi{10.1080/00927872.2022.2093884}; (corrected version: \arxiv{2203.00311}).
\bibitem{b2} A. Ben Hassine, T. Chtioui, M. Elhamdadi, and S. Mabrouk, \emph{Cohomology and deformations of left-symmetric Rinehart algebras}, Communications in Mathematics, 32 (2024), no. 2, 127--152; \doi{10.46298/cm.12285}.
\bibitem{b3} I. Correa and I. Hentzel, \emph{On solvability of noncommutative power-associative nilalgebras}, Journal of Algebra, 240 (2001), no. 1, 98--102; \doi{10.1006/jabr.2000.8708}.
\bibitem{b4} M. Gerstenhaber and H. Myung, \emph{On commutative power-associative nilalgebras of low dimension}, Proceedings of the American Mathematical Society, 48 (1975), 29--32; \doi{10.2307/2040687}.
\bibitem{b5} F. Grunewald and J. O'Halloran, \emph{Varieties of nilpotent Lie algebras of dimension less than six}, Journal of Algebra, 112 (1988), 315--325; \doi{10.1016/0021-8693(88)90093-2}.
\bibitem{b6} A. Fern\'andez Ouaridi, I. Kaygorodov, M. Khrypchenko, and Y. Volkov, \emph{Degenerations of nilpotent algebras}, Journal of Pure and Applied Algebra, 226 (2022), no. 3, 106850; \doi{10.1016/j.jpaa.2021.106850}; (corrected version: \arxiv{1905.05361}).
\bibitem{b7} N. Ismailov, I. Kaygorodov, and Y. Volkov, \emph{Degenerations of Leibniz and anticommutative algebras}, Canadian Mathematical Bulletin, 62 (2019), no. 3, 539--549; \doi{10.4153/S0008439518000073}; (corrected version: \arxiv{1808.00907}).
\bibitem{b8} I. Kaygorodov, \emph{Non-associative algebraic structures: Classification and structure}, Communications in Mathematics, 32 (2024), no. 3, 1--62; \doi{10.46298/cm.14070}; (corrected version: \arxiv{2306.00425}).
\bibitem{b9} I. Kaygorodov, M. Khrypchenko, and S. Lopes, \emph{The geometric classification of nilpotent algebras}, Journal of Algebra, 633 (2023), 857--886; \doi{10.1016/j.jalgebra.2023.06.039}; (corrected version: \arxiv{2102.10392}).
\bibitem{b10} I. Kaygorodov, M. Khrypchenko, and P. P\'aez-Guill\'an, \emph{The geometric classification of non-associative algebras: A survey}, Communications in Mathematics, 32 (2024), no. 2, 185--288; \doi{10.46298/cm.12331}; (corrected version: \arxiv{2410.10825}).
\bibitem{b11} Z. Liu and Y. Sheng, \emph{Omni-representations of Leibniz algebras}, Communications in Mathematical Research, 40 (2024), no. 1, 30--42; \doi{10.4208/cmr.2023-0020}.
\bibitem{b12} S. Lopes, \emph{Noncommutative algebra and representation theory: Symmetry, structure \& invariants}, Communications in Mathematics, 32 (2024), no. 2, 63--117; \doi{10.46298/cm.12248}.
\bibitem{b13} B. Mondal and R. Saha, \emph{Cohomology, deformations and extensions of Rota--Baxter Leibniz algebras}, Communications in Mathematics, 30 (2022), no. 2, 93--117; \doi{10.46298/cm.9801}.
\bibitem{b14} E. Quintero Vanegas and J. Gutierrez Fernandez, \emph{Power associative nilalgebras of dimension 9}, Journal of Algebra, 495 (2018), 233--263; \doi{10.1016/j.jalgebra.2017.10.019}.
\bibitem{b15} A. Smoktunowicz, \emph{A simple nil ring exists}, Communications in Algebra, 30 (2002), no. 1, 27--59; \doi{10.1081/AGB-120006478}.
\end{thebibliography}
\end{document}